\documentclass{amsart}
\usepackage{spverbatim}
\usepackage{breqn}
\usepackage{siunitx}
\usepackage{changepage}
\usepackage{dsfont}
\usepackage{tikz}
\usetikzlibrary{matrix}

\usepackage{amssymb}

\newtheorem{theorem}{Theorem}[section]
\newtheorem{lemma}[theorem]{Lemma}
\newtheorem{prop}[theorem]{Proposition}
\newtheorem{coro}[theorem]{Corollary}

\theoremstyle{definition}

\newtheorem{example}[theorem]{Example}

\theoremstyle{remark}
\newtheorem{remark}[theorem]{Remark}

\numberwithin{equation}{section}

\usepackage{algorithm}
\usepackage{float}
\usepackage[noend]{algorithmic}

\DeclareMathOperator{\Aut}{Aut}
\DeclareMathOperator{\Tr}{Tr}
\DeclareMathOperator{\Gal}{Gal}
\DeclareMathOperator{\trace}{Trace}
\DeclareMathOperator{\range}{range}
\DeclareMathOperator{\rank}{rank}

\DeclareMathOperator{\conv}{conv}

\newcommand{\Sym}{{\mathcal S}}

\newcommand{\C}{{\mathbb C}}

\newcommand{\eLL}{{\mathcal L}}

\newcommand{\N}{{\mathbb N}}
\newcommand{\Q}{{\mathbb Q}}
\newcommand{\R}{{\mathbb R}}
\newcommand{\Z}{{\mathbb Z}}

\newtheorem{thm}{Theorem}

\newtheorem{predefn}[thm]{Definition}

\newtheorem{prerem}[thm]{Remark}

\newtheorem{preexample}[thm]{Example}

\title{Facial reduction for exact polynomial sum of squares decompositions}

\author{Santiago Laplagne}
\address{Departamento de Matem\'atica, FCEN, Universidad de Buenos Aires - Ciudad Universitaria,
Pabell\'on I - (C1428EGA) - Buenos Aires, Argentina}
\email{slaplagn@dm.uba.ar}

\begin{document}
\begin{abstract}
We study the problem of decomposing a non-negative polynomial as an exact sum of squares (SOS) in the case where
the associated semidefinite program is feasible but not strictly feasible (for example if the polynomial has real zeros). Computing symbolically roots of the original polynomial and applying facial reduction techniques, we can solve the problem algebraically or restrict to a subspace where the problem becomes strictly feasible and a numerical approximation can be rounded to an exact solution.

As an application, we study the problem of determining when can a rational polynomial that is a sum of squares of polynomials with real coefficients be written as sum of squares of polynomials with rational coefficients, and answer this question for some previously unknown cases. We first prove that if $f$ is the sum of two squares with coefficients in an algebraic extension of $\Q$ of odd degree, then it can always be decomposed as a rational SOS. For the case of more than two polynomials we provide an example of an irreducible polynomial that is the sum of three squares with coefficients in $\Q(\sqrt[3]{2})$ that cannot be decomposed as a rational SOS.
\end{abstract}
\maketitle

\section{Introduction} \label{secIntro}
Decomposing a non-negative polynomial on $\R^n$ as a sum of squares is a classical problem.
In 1888, D. Hilbert proved in a famous paper \cite{Hilbert} that every non-negative polynomial in $n$ variables and even degree $d$ can be represented as a sum of squares of other polynomials if and only if either (a) $n = 1$ or (b) $d = 2$ or (c) $n = 2$ and $d = 4$. The first explicit example of a non-negative polynomial that cannot be written as a sum of squares was found by T. Motzkin \cite{Motzkin} in 1967. His example is $f = x^4y^2 + x^2y^4 - 3x^2y^2 + 1$.

Recently, the algorithmic problem of determining whether a real polynomial is decomposable as a sum of squares, and if so to compute such a decomposition, has become an active area of research as new tools have been developed to tackle this problem. It provides a certificate of non-negativity and has therefore applications in different areas of mathematics. The strategy usually applied is to pose it as a semidefinite programming problem, a class of convex optimization problems which can be efficiently solved by numerical methods. In this setting, computing a sum of squares (SOS) decomposition is equivalent to finding a positive semidefinite matrix in a space of symmetric matrices.

Using available software packages one can compute a numerical approximation to a solution, when it exists.  In theoretical and practical applications, it is often desirable to obtain exact rational or algebraic solutions. In \cite{parrilo}, H. Peyrl and P. Parrilo study how to construct a rational solution from a numerical approximation. Their method is applicable when the associated semidefinite program is strictly feasible, that is, if it admits a positive definite solution. In that case, they can round the solution to a sufficiently close rational point that preserves the positiveness of the eigenvalues of the associated matrix.

That strategy fails when the SDP is feasible but not strictly feasible. In the presence of null eigenvalues there is no direct way to round the numerical solution assuring that the eigenvalues close to zero become exact zeros. In some cases, for example if the original polynomial has real roots, a positive definite solution cannot exist and hence their method cannot be applied. In our work, using as a starting point a suggestion in that paper, we extend their algorithm reducing the dimension of the underlying search space, a technique called facial reduction. Combining the restrictions imposed by the real zeros with other restrictions given by the properties of the solution we are looking for, we are able to solve the problem in cases where the previous algorithm could not succeed.

As an application, we study the problem of determining when can a rational polynomial that is a sum of squares of polynomials with coefficients in an algebraic extension $K$ of $\Q$ be written as sum of squares of polynomials with rational coefficients. This question was originally raised by B. Sturmfels. It is a natural question in the context of polynomial optimization where one is interested in knowing in which cases numerical solutions can be rounded to provide exact certificates.

A first general result for this problem was given by C. Hillar \cite{hillar}, who showed that the question has a positive answer when $K$ is totally real. In \cite{scheiderer}, C. Scheiderer gave the first family of counterexamples of polynomials that are the sum of squares of polynomials with real coefficients that cannot be decomposed as rational SOS. The examples in his work are obtained by multiplying pairs of conjugate linear forms and are hence always sum of two squares over an extension $K$ of $\Q$ of even degree. In Section 5 of the referred paper, he posed some open questions. We are interested in Question 5.3: let $K$ be a number field of odd degree, and assume that a form $f$ is a sum of squares over $K$. Then is $f$ a sum of squares over $\Q$?
In Section \ref{section:sumoftwo} we prove that this is indeed the case if $f$ is the sum of two squares. In Section \ref{section:counterexample} we provide a negative answer to this question in the general case, providing an explicit polynomial that is the sum of three squares of polynomials with coefficients in an extension of $\Q$ if degree 3 that cannot be written as a sum or rational squares. The counterexample we construct gives also a negative answer to Question 5.1: Are there examples that are irreducible over $\C$? The examples Scheiderer constructs are products of linear forms and this simple structure allows him to prove his results. For the case of irreducible polynomials there were no tools available to tackle the problem, and this is what made the question relevant.

For the implementation of our algorithms, we need to be able to work symbolically with roots of systems of polynomials. For that purpose we use Maple software \cite{Maple} which can work efficiently in algebraic extension of the rational numbers. Linking Maple with Matlab \cite{Matlab} and SEDUMI \cite{SEDUMI} (a software for optimization over symmetric cones) we combine symbolic and numerical methods to get a very general and easy to use software package which we call \texttt{rationalSOS}.

The main contributions in this paper are: algorithms to write a non-negative polynomial with real roots as an exact sum of squares (Section \ref{section:facialreduction}) together with a Maple package implementation of the algorithms; a proof that all rational polynomials that can be written as sum of \emph{two} polynomials with coefficients in an algebraic extension of $\Q$ of odd degree can be written as a sum of squares of polynomials with rational coefficients (Section \ref{section:sumoftwo}); and new examples of non-negative polynomials with rational coefficients that are sum of squares of polynomials with non-rational coefficients that cannot be written as sum of squares of polynomials with rational coefficients (Section \ref{section:counterexample}).

\section{Preliminaries and notation}
\subsection{Sum of squares (SOS) decomposition }
Let $K[x] = K[x_1, \dots, x_n]$ be the ring of polynomials with coefficients in a field $\Q \subseteq K \subseteq \R$. We use the multi-index $\alpha \in \Z^n_+$ to denote the monomial $x^\alpha = x_1^{\alpha_1} \dots x_n^{\alpha_n}$ and define the degree of $x^\alpha$ to be $|\alpha| = \alpha_1 + \dots + \alpha_n$.  A polynomial $f \in K[x_1, \dots, x_n]$ that is non-negative over $\R^n$ can always be homogenized adding a new variable without modifying the non-negativity, and any SOS decomposition for one case can be easily modified to the other case, therefore in this paper we will always work with homogeneous polynomials, which are usually called forms.

For a polynomial $f \in K[x]$, we note $f_\alpha$ the coefficient of $x^\alpha$ in $f$. Let $M_d$ denote the set of exponents of monomials of degree $d$: $M_d = \{\alpha \in \Z^n_{\ge 0}: |\alpha| = d\}$. We note $m_d := |M_d|$, or simply $m$ when $d$ is clear, the number of monomials of degree $d$.

We say that a non-negative $f \in K[x]$ admits a SOS (sum of squares) decomposition if it can be written as $f = \sum_{i=1}^r p_i^2$ for some $p_i \in \R[x]$, $1 \le i \le r$, and we say that $f$ admits a rational SOS decomposition if we can take $p_i \in \Q[x]$, $1 \le i \le r$. This is equivalent to the existence of  $c_i > 0 \in \Q$ and $p_i \in \Q[x]$, $1 \le i \le r$, such that
\[
f = \sum_{i=1}^r c_i p_i^2,
\]
because any rational number $p / q$ can be written as the sum of $pq$ times $(1/q)^2$ (it can however change the number of polynomials required for such decomposition). We will usually work with the latter decomposition in this paper, which is more consistent with the matrix factorization tools we will use.

\begin{remark}
Let $C(f) := \conv(\{\alpha \mid f_\alpha \neq 0\})$ be the convex hull of the exponents set of $f$. B. Reznick proved in \cite{Reznick} that only monomials with exponents contained in $\frac{1}{2} C(f)$ can appear in a SOS decomposition. In particular, if $f$ is homogeneous of degree $2d$, then the polynomials $p_i$, $1 \le i \le r$, in the SOS decomposition must be homogeneous polynomials of degree $d$.
\end{remark}

\subsection{The semidefinite programming (SDP) formulation}
\label{sub:sdp}
Let $\Sym^m \subset \R^{m \times m}$ be the space of symmetric matrices. A matrix $A \in \Sym^m$ is called \emph{positive semidefinite}  if $v^t A v \ge 0$, $\forall v \in \R^m$, and it is called \emph{positive definite} if $v^t A v > 0$, $\forall v \in \R^m \setminus \{0\}$.  We call $\Sym^m_0$ and $\Sym^m_+$ the subsets of $\Sym^m$ of positive semidefinite and positive definite matrices respectively.

Starting with a form $f \in \R[x_1, \dots, x_n]$ of degree $2d$ we first recall how to pose the problem of computing a SOS decomposition as a semidefinite programming problem. We express the given polynomial as a quadratic form
\begin{equation}
\label{SDP}
f(x) = v(x)^t Q v(x),
\end{equation}
where $v(x)$ is the vector of all monomials of degree $d$, i.e., $v_\alpha(x) =x^\alpha$, $\alpha \in M_d$, and $Q \in \Sym^{M_d}$ (the set of symmetric matrices with rows and columns indexed by the elements in $M_d$).
Since the components of $v(x)$ satisfy algebraic relations, $Q$ is in general not unique and formula (\ref{SDP}) gives a set of linear equations for the entries of $Q$. Hence the set of all matrices $Q$ for which (\ref{SDP}) holds is an affine subspace of the set of symmetric matrices.
We denote this affine subspace by
\[
 \eLL := \{Q \in \Sym^{M_d} | f(x) = v(x)^t Q v(x)\},
\]
and $\eLL_\Q := \eLL \cap \Q^{m_d \times m_d}$.
The connection between SOS decompositions and the space $\eLL$ is given by the following lemma (see \cite[p.106]{choi2} and \cite{parrilo}):
\begin{lemma}
\label{SOS-SDP}
Let $f(x) \in \R[x]$. Then $f(x)$ admits a SOS decomposition if and only if there exists a positive semidefinite matrix $Q \in \eLL$.
Moreover, $f(x)$ admits a rational SOS decomposition if and only if there exists a positive semidefinite matrix $Q \in \eLL_\Q$.
\end{lemma}

By this lemma, we can pose the problem as a semidefinite program (SDP), a class of convex optimization problems. The advantage is that SDPs can be efficiently solved by interior point methods, which makes SOS problems computationally tractable. There are several numerical solvers that can compute effectively a positive semidefinite matrix in $\eLL$. We use SEDUMI \cite{SEDUMI} in our implementation.

\begin{example}
\label{example:3x3}
Let $f(x,y) = 10x^4+2x^3y+27x^2y^2-24xy^3+5y^4 \in \R[x,y]$. Solving Equation (\ref{SDP}) we get
\[
f(x,y)=
\begin{pmatrix}
x^2 & xy & y^2
\end{pmatrix}
\begin{pmatrix}
10 & 1 & a \\
1 & -2a + 27 & -12 \\
a & -12 & 5 \\
\end{pmatrix}
\begin{pmatrix}
x^2 \\
xy \\
y^2 \\
\end{pmatrix}
\]
for any $a \in \R$. Hence $\eLL = \{W(a) : a \in \R$\}, with $W(a)$ the $3 \times 3$ matrix in the identity above. There existes a SOS decomposition of $f$ if and only if there exists $a \in \R$ such that $W(a)$ is positive semidefinite.
\end{example}


H. Peyrl and P. Parrilo \cite{parrilo} studied the problem of computing exact rational SOS decompositions.
If the SDP problem is strictly feasible, that is, if it admits a positive definite solution, a numerical solver that maximizes the value of the minimum eigenvalue will compute a positive definite solution matrix if it exists.
In that case, one expects that the approximation computed can be rounded to a close enough rational solution that does not change the positivity of the eigenvalues.
 If the problem is feasible but not strictly feasible, that is, it admits a positive semidefinite matrix solution but not a positive definite one, then rounding the approximate solution is likely to introduce negative eigenvalues.

In our work, to tackle this problem, we add linear restrictions to the search space in order to reduce the dimension. We expect to find a space where the problem becomes strictly feasible and an approximate solution can be rounded to an exact solution. For simplicity, we will work with linear affine parametrizations of the resulting subspaces (also called \emph{image representations} in \cite[Section 3.2]{parrilo}). For example, matrix $W(a)$ in Example \ref{example:3x3} is a parametrization of $\eLL$.


For any parametrization $W(a_1, \dots, a_s)$ of $\eLL$ or a linear subspace, we can compute its rank as a matrix in $\R[a_1, \dots, a_s]^{m_d \times m_d}$. This rank is independent of the parametrization chosen and it is equal to the maximum rank over all possible specializations of the parameters. For short, we will from now on refer to the rank and kernel of a parametrization of $\eLL$ (or a subspace) as the rank and kernel of $\eLL$ (or the subspace). The set of values of the parameters for which the maximum rank is attained is open and dense in the space of all possible values, hence by giving random values to the parameters we can almost always compute the correct rank. This is more efficient than computing the rank of the matrix in $\R[a_1, \dots, a_s]^{m_d \times m_d}$. 

\begin{example}
\label{exampleXYZ}
(See \cite[Worksheet A]{rationalSOS} for the computations in Maple.)
Let
\[
\begin{aligned}
f(x,y,z) = {} & 10x^4+6x^3y-22x^3z+39x^2y^2-24x^2yz + {}\\
& 33x^2z^2-20xy^2z+8xyz^2-20xz^3+25y^4+10y^3z+y^2z^2+4z^4.
\end{aligned}
\]

Solving the equations for the entries of $Q$ from (\ref{SDP}), indexing the columns by $v = (x^2, xy, xz, y^2, yz, z^2)$, we get the parametrization of $\eLL$:

\begin{equation}
\label{exampleXYZ-parametrization}
W=\left( \begin {array}{cccccc} 10&3&-11&a_{{1,4}}&-12-a_{{2,3}}&a_{{1,6}}\\ \noalign{\medskip}3&39-2\,a_{{1,4}}&a_{{2,3}}&0&-10-a_{{3,4}}&4-a_{{3,5}}\\ \noalign{\medskip}-11&a_{{2,3}}&33-2\,a_{{1,6}}&a_{{3,4}}&a_{{3,5}}&-10\\ \noalign{\medskip}a_{{1,4}}&0&a_{{3,4}}&25&5&a_{{4,6}}\\ \noalign{\medskip}-12-a_{{2,3}}&-10-a_{{3,4}}&a_{{3,5}}&5&1-2\,a_{{4,6}}&0\\ \noalign{\medskip}a_{{1,6}}&4-a_{{3,5}}&-10&a_{{4,6}}&0&4\end {array} \right)
\end{equation}

This matrix has rank 6 (it can be easily computed assigning random values to the unknowns).
To find a SOS decomposition of $f$, we look for a positive semidefinite matrix of this shape. We run the numerical solver SEDUMI to compute an approximate solution that maximizes the minimum eigenvalue. The output matrix has the following eigenvalues:
\[
\{\num{-0.5356e-8}, \num{0.2776e-7}, 4.2732, 16.5132, 28.9722, 46.9092\}.
\]
Intuitively, we interpret these values as two null eigenvalues and four positive eigenvalues, that is, a positive semidefinite solution was found. However, there is no direct way to round the solution matrix so that the approximate zeros become exact zeros. That is, an exact SOS decomposition cannot be directly computed from the approximate SOS decomposition. We will show how to handle this case in the following sections, using as a starting point a suggestion given in \cite{parrilo}.
\end{example}

\section{Facial reduction}
\label{section:facialreduction}
Facial reduction is a general technique to simplify SDPs with no strictly feasible solution. It was first introduced in \cite{wolkowicz} and used successfully in many applications. In our case it can be applied in a very nice and simple way. We recall that the principle of facial reduction is based on the following result (the nice graphical presentation is taken from F. Permenter talk \cite{permenter}).

\begin{lemma}
Fix $U \in \R^{m \times \ell}$. Then
\[
\begin{tikzpicture}[x=3in,y=2in]
    \tikzstyle{ann} = [draw=none,fill=none,right]
    \matrix[nodes={draw, thick, fill=none},
        row sep=0.3cm,column sep=0.3cm]
    {
    \node[draw=none,fill=none] (N1) {$\substack{X \in \Sym^m_0  \\ \range X \subset \range U} \iff$}; &
    \node[rectangle, minimum height=0.75in, minimum width = 0.75in] (N2) {$X$}; &
    \node[draw=none] (N3) {$=$}; &
    \node[rectangle, minimum height=0.75in, minimum width = 0.35in] (N4) {$U$}; &
    \node[rectangle, minimum height=0.35in, minimum width = 0.35in,  anchor=north west] (N5) at ([xshift=1cm]N4.north east) {$\hat X$}; &
    \node[rectangle, minimum height=0.35in, minimum width = 0.75in,  anchor=north west] (N6) at ([xshift=1cm]N5.north east) {$U^t$}; &
    \\
    };
\end{tikzpicture}
\]
for some $\hat X \in \Sym_0^\ell$.
\end{lemma}

Assume $X$ has rank $\ell$ and let $\Omega \subset \{1, \dots, m\}$ be a set of indices of $\ell$ linearly independent columns of $X$. There exists $U$ as in the lemma such that it is the identity matrix in the rows indexed by $\Omega$ (apply Gaussian elimination to the columns indexed by $\Omega$). Hence, if $X \in \Sym_0^m$ has rank $\ell \le m$, there exists a positive definite principal submatrix $\hat X \in \Sym_+^\ell$ that satisfies the formula of the lemma.



Our general strategy is then to find linear relations in the space $\eLL \cap \Sym_0^m$ that reduce the rank of the matrices in the search space, so that we can restrict to a principal submatrix of the parametrization where the associated SDP program becomes strictly feasible. We can then compute a positive definite solution using a numerical solver, and plug the values found in the original parametrization matrix to get a positive semidefinite matrix with $\ell$ positive eigenvalues and $m - \ell$ exact null eigenvalues.


\subsection{Real zeros of the polynomial}
The first step for facial reduction is based on the following observation in \cite{parrilo}: let $Q \in \eLL \cap \Sym_0^m$ and $x^\star \in \R^n$ be such that $f(x^\star) = 0$, then $Q v(x^\star) = 0$. That is $v(x^\star)$ is in the kernel of  $\eLL \cap \Sym_0^m$. This follows from the following more general property:

\begin{prop}
\label{prop:nulleq}
If $v \in \R^m$, $Q \in \Sym_0^m$ and $v^t Q v = 0$ then $Q v = 0$.
\end{prop}

Hence, exact real solutions of $f(x) = 0$ provide linear restrictions to the search space that can reduce its dimension.

\begin{remark}
For our facial reduction strategy we need to compute real solutions symbolically.
Since we are assuming the polynomial to be non-negative, the real solutions must be local minima and we can look for them by solving the system
\[
\left\{f(x) = 0, \frac{\partialº f}{\partial x_1}(x) = 0, \dots, \frac{\partial f}{\partial x_n}(x) = 0\right\}.
\]

In our implementation in Maple we use the procedure \texttt{solve} to obtain real solutions, which computes parametrizations in algebraic extensions of $\Q$ of different branches of solutions, and we look for real points in them. Alternatively, one can compute the prime components of the above system using primary decomposition algorithms, and look for real points in them. Computing these exact solutions can be the bottle neck in some high degree examples. However, we will see in Section \ref{section:counterexample} and Appendix \ref{appendix:usage} that using these techniques we can solve many interesting problems which are otherwise difficult to handle.
\end{remark}

\subsection{Algebraic conjugates of the real zeros}
Let $\bar\Q$ be the algebraic closure of $\Q$ and let $x^\star \in (\bar\Q \cap \R)^n$ be an algebraic real zero of $f$. An important observation when we are looking for \emph{rational} SOS decompositions, is that for $L$ a minimal algebraic extension of $\Q$ containing all the entries of $v(x^\star)$ and $G$ a splitting field of $L$, if $Q \in \eLL \cap \Sym_0^m$ has rational entries,
\[
0 = f(x^\star) \Rightarrow 0 = Q v(x^\star) \Rightarrow 0 = \sigma(Q v(x^\star)) = Q  \sigma(v(x^\star)),
\]
where $\sigma \in \Aut(G/\Q)$ and abusing the notation, $\sigma(v) = (\sigma(v_1), \dots, \sigma(v_m))$ for $v \in L^m$. That is, $\sigma(v(x^\star)) \in \ker(Q)$ for any $\sigma \in \Aut(G/\Q)$, and so is any linear combination of these elements.
Of special interest is the following special case.

\begin{prop}
Let $x^\star \in (\bar\Q \cap \R)^n$. Let $L$ be a minimal extension of $\Q$ containing all the entries of $v(x^\star)$.
Assume $Q \in \eLL_\Q \cap \Sym_0^m$. Then $\trace_{L/\Q}(v(x^\star)) \in \Q^m$ is in the kernel of $Q$.
\end{prop}

The importance of this proposition is that it provides a vector with rational coordinates in the kernel of $\eLL_\Q \cap \Sym_0^m$, and hence the linear relations derived are simple and easy to handle. One could alternatively work in the splitting field $G$ of $L$ and add all the equations for the conjugates, but computing a splitting field of $L$ can be very expensive even for small degree extensions. In contrast, the trace of elements in $L$ can be easily obtained from the coefficients of the minimal polynomial of a primitive element for the extension $L / \Q$.

\begin{example}
Carrying on Example \ref{exampleXYZ}, we look for real solutions of $f(x,y,z) = 0$.
Using the command \texttt{solve} from Maple to solve the system
\[
\left\{f(x,y,z) = 0, \frac{\partialº f}{\partial x}(x,y,z) = 0, \frac{\partial f}{\partial y}(x,y,z) = 0, \frac{\partial f}{\partial z}(x,y,z) = 0\right\}.
\]
we get the solution:
\[
\gamma(t) = \left(t, \alpha t,  \frac{1}{46} (\alpha^3t+128\alpha^2+25\alpha+73) \right)
\]
where $\alpha$ is a root of $m(Z) = 50Z^4+28Z^3-Z^2+23Z-8$. This polynomial has two real roots.

We fix $t=1$ and use this solution to reduce the dimension of the search space as explained above.  That is, we define $\eLL_1 = \eLL \cap \{Q\in \R^{6\times 6} :Q . v(s(1)) = 0\}$. This is equivalent to solving the system of linear equations given by $W.v(s(1)) = 0$, where $W$ is the parametrization of $\eLL$ given in Example \ref{exampleXYZ}. Solving these equations we get that the dimension of $\eLL_1$ is 3 and the rank of the parametrization matrix is 5.

 We could now take a principal submatrix of the parametrization of $\eLL_1$ of size and rank 5 and use SEDUMI to compute numerically a positive semidefinite matrix in the shape of that principal submatrix. The problem is not strictly feasible, the solution we get has still an approximate null eigenvalue and we cannot use it to compute an exact solution to the problem.

If instead of using the equation $W . v(s(1))=0$ we compute the trace of $v(s(1))$ for the extension $\Q(\alpha) \hookrightarrow \Q$, $\tilde v = (4, -14/25, 53/10, 221/625, -396/125, 1209/100)$, and we solve $W .\tilde v = 0$ for $W$ the parametrization of $\eLL$, we find that there is only one solution:
\[
  \left( \begin {array}{cccccc} 10&3&-11&15&3&2\\ \noalign{\medskip}3&9&-15&0&0&6\\ \noalign{\medskip}-11&-15&29&-10&-2&-10\\ \noalign{\medskip}15&0&-10&25&5&0\\ \noalign{\medskip}3&0&-2&5&1&0\\ \noalign{\medskip}2&6&-10&0&0&4\end {array} \right)
\]

Diagonalizing this matrix, we obtain a rational solution to our original problem:
\[
f = (x^2 + 3xy - 5xz+ 2z^2)^2 + (3x^2 - 2xz +yz+5y^2)^2
\]
(the decomposition obtained from the matrix is not unique, different diagonalizations can give different outputs).

\end{example}


\subsection{Ghost solutions}
In the last subsection we used real zeros $x^\star$ of $f$ to obtain elements $v(x^\star)$ in the kernel of $\eLL \cap \Sym_0^m$. Based on Remark \ref{prop:nulleq}, we can in many cases obtain other elements that must be in that kernel.

{\bf Null entries in the diagonal.} If a parametrization matrix $W$ of $\eLL$ (or a subspace of $\eLL$ obtained after intersecting with some linear restrictions) has a null value in the diagonal, this is an obstruction for the matrices in $\eLL$ being positive definite, all the entries in the diagonal of a positive definite matrix are positive. This is an easy to remove obstruction by the following simple lemma.

\begin{lemma}
If $Q \in \Sym_0^{m \times m}$ and $Q_{kk} \equiv 0$ then $e_k^t . Q . e_k$ = 0 and therefore $Q . e_k = 0$, where $e_k$ is the $k$-th canonical vector.
\end{lemma}

Hence, $W_{kk} = 0$ gives $m-1$ new equations to reduce the dimension of the search space. Note that all the entries in the $k$-th row and column of the parametrization of $\eLL \cap \{Q \in \R^{m\times m}: Q e_k = 0\}$ become zero and we can ignore that row and column.

In practice, it is a common situation that a parametrization matrix has null values in the diagonal and so this is a very simple and useful technique in the facial reduction strategy.

Note that there could be no real point $x = (x_1, \dots, x_n)$ such that $v(x) = e_{k}$. That is, we may be using solutions of $v^t Q v = 0$ that do not come from solutions of $f = 0$.

{\bf Principal submatrices with zero determinant.} The case of the last paragraph can be extended to any principal submatrix of size $\ell < m$ with zero determinant.
\begin{lemma}
Let $Q \in \Sym_0^{m \times m}$ and $\Omega \subset \{1, \dots, m\}$ be such that $\det(Q_\Omega) = 0$, where $Q_\Omega$ is the principal submatrix of $Q$ with rows and columns indexed by $\Omega$. Let $\tilde v \in \R^\ell$, $\ell = |\Omega|$, be a vector in $\ker(Q_\Omega)$. Let $v \in \R^m$ be the extension of $\tilde v$ to $\R^m$ with zeros in the other entries (that is, $v_{i_j} = \tilde v_j$ if $i_j \in \Omega$ and $v_i = 0$ if $i \not\in \Omega$). Then $Q . v = 0$.
\end{lemma}

 To apply this lemma, we look for principal submatrices of a parametrization $W$ of the search space with zero determinant and derive the corresponding linear restrictions. In practice, searching among all possible principal submatrices for those with zero determinant can be time consuming, so in our implementation we consider only $2 \times 2$ principal submatrices. This usually provides new conditions and decrease the rank of $\eLL$.

Note again that since a parametrization $W$ of $\eLL$ is a matrix with parameters, if we take large submatrices it can be more efficient to replace the parameters by random values to determine whether the determinant is zero or not. Also, in theory it could happen that the vector $\tilde v$ in the kernel contains parameters of $W$, and hence it cannot be used to get linear relations. However this theoretical situation never occurred in our examples.

\subsection{Rational entries.}
\label{subsection:rational}
The last tool we use to reduce the dimension of the space when we are looking for rational solutions is to force all coefficients of non-rational numbers in the entries of a parametrization matrix to be $0$. For example, if an entry of the matrix is $\sqrt{2} t$ and another one is $\sqrt{3} t$, these two entries cannot be rational at the same time, unless $t = 0$, hence in this case we can impose $t = 0$.  Since the entries of the matrix can be more complicated with many unknowns and different coefficients, it can be difficult to derive these type of restrictions. We simplify this restriction as follows. Let $a = a_0 + a_1 t_1 + \dots + a_s t_s$, $a_i \in \bar \Q$ for $0 \le i \le s$, be an entry of the matrix and $\alpha$ a primitive element for the extension $\Q(a_0, a_1, \dots, a_s) \hookrightarrow \Q$. That is, we can regard the elements $a_i$ as polynomials $a_i(\alpha)$ with rational coefficients. Now we force all the coefficients of the positive powers of $\alpha$ in $a$ to be zero, and we use the resulting equations to reduce the dimension of the problem. For example if $a = 1 + (\sqrt[3]{2} + 2 \sqrt[3]{2}^2) t_1 + (\sqrt[3]{2} + \sqrt[3]{2}^2) t_2$, we add the restrictions $t_1 + t_2 = 0$ and $2 t_1 + t_2 = 0$, from which we derive $t_1 = t_2 = 0$. Note that the restrictions we add are not in general necessary conditions, even in the case of rational solutions, and hence we can be losing solutions. This strategy turns out to be very effective in practice to find rational solutions, but it cannot be used if we are trying to find all solutions or prove that there are no rational solutions. Looking more carefully at the coefficients and the relations among them we can use this strategy to derive necessary conditions, but we do not carry further this analysis in this paper.

\begin{example}
\label{exampleB}
(See \cite[Worksheet B]{rationalSOS} for the computations in Maple.)
We apply the above techniques in the following example:
\[
\begin{aligned}
f(x,y,z) = {} & 3618x^8+468x^7y+6504x^7z-1104x^6y^2+2616x^6yz+57481x^6z^2- \\
& 144x^5y^3-1652x^5y^2z-16440x^5yz^2+23420x^5z^3+160x^4y^4+ \\
& 1392x^4y^3z-2520x^4y^2z^2- 28448x^4yz^3+91320x^4z^4-240x^3y^4z+ \\
& 1728x^3y^3z^2+10524x^3y^2z^3-85500x^3yz^4+34740x^3z^5-3696x^2y^3z^3+ \\
& 28920x^2y^2z^4- 15192x^2yz^5-57267x^2z^6+720xy^4z^3-3312xy^3z^4- \\
& 3168xy^2z^5+ 26352xyz^6-40176xz^7+720y^4z^4+864y^3z^5-9072y^2z^6+ \\
& 46656z^8.
\end{aligned}
\]
This polynomial is a sum of three squares of polynomials with coefficients in $\Q(\alpha)$ with $\alpha$ a root of $Z^3-2$ and we want to see if we can write it as a sum of squares of polynomials with rational coefficients (we will show in Section \ref{section:counterexample} how this polynomial was constructed).

Since $f$ is homogenous of degree $8$ in $x, y, z$, the parametrization matrix $W$ of the space $\eLL$ of solutions to (\ref{SDP}) is a $15 \times 15$ matrix with rows and columns indexed by $v(x,y,z) = (x^4, x^3y, x^3z, x^2y^2, x^2yz,  x^2z^2,  xy^3, xy^2z, \allowbreak xyz^2, xz^3, \allowbreak y^4, y^3z, \allowbreak y^2z^2, yz^3, z^4)$. The space $\eLL$ has dimension 75 and rank 15.

If we use the numeric solver SEDUMI to compute a positive semidefinite matrix in $\eLL$ we get a solution with 10 positive eigenvalues and 5 (approximate) null eigenvalues. We need to reduce the generic rank of the parametrization matrix to 10 so that we can find a $10 \times 10$ principal submatrix of full rank, use the numerical solver to compute a positive definite solution in that space, and extend the solution to a $15 \times 15$ matrix with $10$ positive eigenvalues and $5$ exact zero eigenvalues.

We solve the system of equations
\[
\left\{f(x,y,z) = 0, \frac{\partial f}{\partial x}(x,y,z) = 0, \frac{\partial f}{\partial x}(x,y,z) = 0, \frac{\partial f}{\partial x}(x,y,z) = 0\right\}
\]
in Maple and obtain the following three branches of solutions:
\begin{align*}
s_1(t) &= (0, t, 0), \\
s_2(t) &= (0, -3t, t), \\
s_3(t) &= (t, \frac{1}{60}(648\alpha^4-327\alpha^3+152\alpha^2-777\alpha-36)t, \alpha t),
\end{align*}
where $\alpha$ is a root of $648Z^5-327Z^4+152Z^3-921Z^2-36Z+36$. This last polynomial has odd degree so it contains at least one real root.

From the first solution we derive the condition $W . v(0,1,0) = 0$ and from the second solution we derive the condition $W . v(0,-3,1) = 0$.
For the last solution we derive first the trace condition for $t = 1$, $W . \Tr(v(s_3(1))) = 0$ These restrictions reduce the problem to $\eLL_1 \subset \eLL$ of dimension 39 and rank 12.

Now we use the ghost solutions coming from null values in the diagonal of the parametrization of $\eLL_1$. There are three null values, and adding the corresponding restrictions we reduce to $\eLL_2 \subset \eLL_1$ of dimension 23 and rank 10.

If we solve numerically this problem, the solution we get has 5 exact null eigenvalues as expected, but also 4 approximate null eigenvalues.
We need to further reduce the dimension of the problem to get an exact solution.

We use the third solution with $t = 1$ and reduce the problem to $\eLL_3$ with dimension 16 and rank 9. Now the parametrization matrix contains non-rational coefficients. To look for rational solutions, we can apply the strategy in Section \ref{subsection:rational}. This reduces the search space to $\eLL_4$ of dimension 4 and rank 6.
We choose a $6 \times 6$ principal submatrix of full rank of the parametrization of $\eLL_4$ and use the numeric solver SEDUMI to compute a positive semidefinite matrix in that space. The problem is now strictly feasible, the solver finds a positive definite solution. From this solution, we obtain a solution to the original problem that has 9 exact null eigenvalues and 6 positive eigenvalues (the rank was 6), hence we have solved the problem.

\begin{remark}
If we compute the characteristic polynomial of the resulting matrix, we obtain a polynomial of the form $\chi(x) = x^9 (x^6-a_5 x^5 + a_4 x^4 - a_3 x^3 + a_2 x^2 - a_1 x + a_0)$, with $a_i \in \Q_{>0}$. Since the second factor has alternating signs by Descarte's rule we know that it has exactly 6 positive roots.
\end{remark}

\end{example}

\subsection{A step by step algorithm}

The ideas developed above are to be considered more as a set of tools to handle the problem than as a fixed algorithm. Depending on the problem, some of the tools will provide results where others may not, and in some cases applying repeatedly some of the tools can speed up the computations. In particular, it seems to be a good strategy to add the equations corresponding to null values in the diagonal as soon as such a value appears.

Nevertheless, for clarity and definiteness we present in this section a step by step algorithm. This is essentially the algorithm that is implemented in procedure \texttt{exactSOS} of the Maple package \texttt{rationalSOS}. For simplicity, we assume that we are looking for exact \emph{rational} solutions and we use only the equations coming from the trace of real solutions, as the plain equations (without taking trace) may introduce non-rational coefficients.

\begin{algorithm}
\label{mainAlgorithm}
\caption{Exact sum of squares decomposition}
\begin{algorithmic}
\REQUIRE $f \in \Q[x_1, \dots, x_n]$ homogenous of degree $2d \in \N$.
\ENSURE $\{p_1, \dots, p_s\} \subset \Q[x_1, \dots, x_n]$ homogenous of degree $d$ and $\{c_1, \dots, c_s\} \subset \Q$ such that $f = c_1 p_1^2 + \dots + c_s p_s^2$, if such decomposition is found. $0$ if no decomposition is found.
\STATE $S := \{\omega_i = (\omega_{i,1}, \dots, \omega_{i,n})\}_{\{1 \le i \le t\}} \subset \bar\Q^n$, a set of solutions of $f = 0$ (which can be computed by solving the system $\{f= 0\} \cap \{\frac{\partial f}{\partial x_i} = 0\}_{1 \le i \le n}$)
\STATE $v := $ vector of all the monomials of degree $d$ in $n$ variables, $v_\alpha = x^\alpha$ for $\alpha \in M_d$
\STATE $W := \text{ matrix parametrization of  } \eLL = \{Q \in \Sym^{S_d} | f(x) = v(x)^t Q v(x)\}$
\FOR{i:=1  \TO t}
\STATE $u := \Tr_{\Q(\omega_{i,1}, \dots, \omega_{i,n})/\Q}(v(\omega_i))$
\STATE Set $\eLL = \eLL \cap \{Q \in \Sym^{S_d} | Q . u = 0 \}$
\STATE $W := \text{ parametrization of } \eLL$, computed by solving the linear equations
\ENDFOR
\STATE Set $\eLL = \eLL \cap \{Q \in \Sym^{S_d} | Q . e_i = 0 \text{ for all } i \text{ such that } W(i,i) = 0\}$
\STATE $W(T_1, \dots, T_k) := \text{ parametrization of } \eLL$
\STATE $s := \rank(W)$
\STATE $W_s := $ principal submatrix of $W$ of rank $s$
\STATE $(t_1, \dots, t_k) := $ numerical solution  to the problem of maximizing the minimal eigenvalue of $W_s(T_1, \dots, T_k)$
\IF{$W_s(t_1, \dots, t_k)$ is positive definite}
\STATE Compute an \emph{LU decomposition} of $W(t_1, \dots, t_k)$: $P, D, U \in \Q^{m_d \times m_d}$ with $P$ a permutation matrix, $U$ upper triangular with all entries in the diagonal equal to $1$, and $D \in \Q_{\ge 0}^{m_d \times m_d}$ diagonal matrix such that $W(t_1, \dots, t_k) = P U^t D U$ (by Gaussian elimination)
\FOR{i:=1  \TO s}
\STATE $p_i = \sum_{j=1}^r U_{i, j} v_j(x_1, \dots, x_n)$
\ENDFOR
\RETURN $\{p_1, \dots, p_s\}$, $\{D_{1,1}, \dots, D_{s,s}\}$
\ELSE
\RETURN 0
\ENDIF
\end{algorithmic}
\end{algorithm}

\section{Sum of squares of two polynomials over a field of odd degree}
\label{section:sumoftwo}

We now turn to the problem of determining when does a rational polynomial that allows a $\R$-SOS decomposition allow also a $\Q$-SOS decomposition.

In \cite[Section 5]{scheiderer} the author poses some open questions. Question 5.3 is: ``Let $K$ be a number field of odd degree, and assume that a form $f$ over $\Q$ is a sum of squares over $K$. Then, is $f$ a sum of squares over $\Q$?'' We provide in this section an affirmative answer for the case of a sum of two squares over $K$.

\begin{theorem}
\label{theorem:twopolys}
Let $f \in \Q[x_1, \dots, x_n]$ be such that
\[
f=p_1^2 + p_2^2,
\]
with $p_1, p_2 \in \Q(\alpha)[x_1, \dots, x_n]$ and $\Q(\alpha) \hookrightarrow \Q$ an algebraic extension of odd degree.

There exist $q_1, q_2 \in \Q[x_1, \dots, x_n]$ such that
\[
f= q_1^2 + q_2^2.
\]
\end{theorem}

For the proof we need the following result \cite[Theorem A]{choi}:
\begin{theorem}
\label{theoremA}
Let $A$ be a unique factorization domain (UFD) and $F$ the quotient field of $A$. Assume that $I = \sqrt{-1} \not\in F$, and that $A[I] = \{a_1 + a_2 I : a_1, a_2 \in A\}$ is also a UFD. If $a \in A$ is a sum of two squares in $F$ then $a$ is a sum of two squares in $A$.
\end{theorem}

Taking $A = \Q[x_1, \dots, x_n]$ we get:

\begin{coro}
\label{theoremC}
Let $f \in \Q[x_1, \dots, x_n]$ and assume that $f$ is a sum of two squares in $\Q(x_1, \dots, x_n)$. Then it is already a sum of two squares in $\Q[x_1, \dots, x_n]$.
\end{coro}

\begin{proof}[Proof of Theorem \ref{theorem:twopolys}]
Let $m(Z) \in \Q[Z]$ be the minimal polynomial of $\alpha$ over $\Q$.
Let $L$ be a splitting field of $m(Z)$ over $\Q$ and let $\alpha = \alpha_0, \alpha_1, \dots, \alpha_{d-1}$ be the roots of $m(Z)$ in $L$. Let $\sigma_i$, $0 \le i \le {d-1}$, be elements of $\Gal(L / \Q)$ such that $\sigma_i(\alpha) = \alpha_i$.

We can factorize $f = p_1^2 + p_2^2 = (p_1 + I p_2)(p_1 - I p_2)$, where $I^2 = -1$. For each $0 \le i \le d-1$, we extend $\sigma_i$ to an automorphisms of $L(I)$, and we call the extension also $\sigma_i$. Taking the product of the conjugates:

\[
f^d = \prod_{i=0}^{d-1}  \sigma_i\left((p_1+Ip_2)(p_1-Ip_2)\right) = \prod_{i=0}^{d-1}  \sigma_i(p_1+Ip_2) \sigma_i(p_1-Ip_2).
\]

Since $\sigma_i(I)$ is either $I$ or $-I$,  $\sigma_i(p_1+Ip_2) \sigma_i(p_1-Ip_2) =  (\sigma_i(p_1)+I\sigma_i(p_2)) (\sigma_i(p_1)-I\sigma_i(p_2))$ and we can regroup the factors:
\[
f^d = \left(\prod_{i=0}^{d-1}  \sigma_i(p_1)+I\sigma_i(p_2)\right) \left(\prod_{i=0}^{d-1}  \sigma_i(p_1)-I\sigma_i(p_2)\right)
\]
which we can rewrite as
\[
f^d = \left(  P_1+I P_2\right) \left( P_1-I P_2 \right),
\]
with $P_1, P_2$ symmetric polynomials on the $d$ conjugates of $\alpha$.
Hence $P_1$ and $P_2$ are fixed by $\Gal(L  / \Q)$ and they are therefore polynomials in $\Q[x_1, \dots, x_n]$.

We obtain
\[
f = \frac{P_1^2}{f^{d-1}} + \frac{P_2^2}{f^{d-1}}
\]
with $d-1$ even. Hence $f$ if the sum of the squares of two rational functions.
By Corollary \ref{theoremC}, this implies that $f$ can be written as the sum of two polynomials with rational coefficients.

\end{proof}

\begin{remark} The polynomials $P_1$ and $P_2$ can be constructed following the above steps (assuming we are able to compute in the splitting field of $\Q(\alpha)$). The proof of Theorem A in \cite{choi} is also constructive. The above proof is then an algorithm for computing the $\Q$-SOS decomposition.
\end{remark}

\begin{example}
(See \cite[Worksheet C]{rationalSOS} for the computations in Maple.)
Let
\begin{align*}
p_1(x,y,z) &= 2\alpha^2x^2z-2\alpha^2 yz^2+2\alpha x^3-2\alpha xyz-2x^3+x^2y-2xz^2+y^3, \\
p_2(x,y,z) &= 2\alpha^2x^3-2\alpha^2 xyz+2\alpha x^2y-2\alpha y^2z-x^3+2x^2z-xy^2+2z^3,
\end{align*}
with $\alpha$ a root of $Z^3-2$.
Let $f(x,y,z) = p_1^2 + p_2^2 = 5x^6+12x^5y+12x^5z+3x^4y^2+12x^4z^2-4x^3y^3-36x^3y^2z-36x^3yz^2-4x^3z^3+3x^2y^4+12x^2z^4+12xy^3z^2+12xy^2z^3+y^6+4z^6 \in \Q[x,y,z]$ (see Section \ref{section:counterexample} for the details on how this example was constructed).

Let $\sigma_i$ be as in Theorem \ref{theorem:twopolys}. Recalling that $P_1$ and $P_2$ are defined so that $P_1 + I P_2 = \prod_{i=0}^{2} (\sigma_i(p_1)+I\sigma_i(p_2))$, we get
\begin{dgroup*}
\begin{dmath*}
P_1 = \sigma_0(p_1)  \sigma_1(p_1)  \sigma_2(p_1) - \sigma_0(p_1)  \sigma_1(p_2)  \sigma_2(p_2) - \sigma_0(p_2)  \sigma_1(p_1)  \sigma_2(p_2) - \sigma_0(p_2)  \sigma_1(p_2)  \sigma_2(p_1)
\end{dmath*}
\begin{dmath*}
P_2 = \sigma_0(p_2)  \sigma_1(p_1)  \sigma_2(p_1) + \sigma_0(p_1)  \sigma_1(p_2)  \sigma_2(p_1) + \sigma_0(p_1)  \sigma_1(p_1)  \sigma_2(p_2) - \sigma_0(p_2)  \sigma_1(p_2)  \sigma_2(p_2)
\end{dmath*}
\end{dgroup*}
which in this example gives

\begin{dgroup*}
\begin{dmath*}
P_1 = -10x^9-39x^8y-24x^8z-42x^7y^2-36x^7yz+6x^7z^2+4x^6y^3+72x^6y^2z+
108x^6yz^2+80x^6z^3+18x^5y^4+120x^5y^3z+126x^5y^2z^2+12x^5yz^3+48x^5z^4-
6x^4y^5-36x^4y^3z^2-240x^4y^2z^3-252x^4yz^4-24x^4z^5-6x^3y^6-36x^3y^5z-
54x^3y^4z^2-40x^3y^3z^3+64x^3z^6+84x^2y^3z^4+72x^2y^2z^5-12x^2yz^6+
18xy^6z^2+12xy^5z^3+24xz^8+y^9+4y^3z^6,
\end{dmath*}
\begin{dmath*}
P_2 = 5x^9+12x^8y+42x^8z-12x^7y^2+72x^7yz+84x^7z^2-40x^6y^3-54x^6y^2z-
36x^6yz^2+58x^6z^3-6x^5y^4-24x^5y^3z-252x^5y^2z^2-240x^5yz^3-36x^5z^4+
12x^4y^5+126x^4y^4z+120x^4y^3z^2+18x^4y^2z^3+48x^4z^5-8x^3y^6+80x^3y^3z^3+
108x^3y^2z^4+72x^3yz^5+12x^3z^6+6x^2y^6z-36x^2y^5z^2-42x^2y^4z^3-3xy^8-
24xy^3z^5-36xy^2z^6-2y^6z^3-8z^9,
\end{dmath*}
\end{dgroup*}
and we verify in Maple that $f^3 = P_1^2 + P_2^2$.

The last step is to convert the expression $f = \frac{P_1^2}{f^2} + \frac{P_2^2}{f^2}$ into a sum of squares of polynomials in $\Q[x,y,z]$, as guaranteed by Corollary \ref{theoremC}. In this case this is immediate, since $P_1$ and $P_2$ are divisible by $f$.

We get
\[
f(x,y,z) = (2x^3+3x^2y-6xz^2-y^3)^2 + (x^3+6x^2z-3xy^2-2z^3)^2.
\]

\end{example}

\section{Sum of squares of three polynomials over a field of odd degree}
\label{section:counterexample}
In this section we provide a negative answer to \cite[Question 5.3]{scheiderer} stated above for the case of a sum of squares of three polynomials over an algebraic extension of $\Q$  of odd degree (which trivially implies a negative answer for any larger number of polynomials). We exhibit a polynomial that can be decomposed as the sum of three squares of polynomials with coefficients in an algebraic extension $\Q(\alpha)$ of $\Q$ of odd degree that cannot be decomposed as sum of squares of polynomials with coefficients in $\Q$. The example also provides an affirmative answer to the first question in \cite[Question 5.1]{scheiderer}: ``Are there examples of forms in $\Q[x_1, \dots, x_n]$ that are sums of squares over $\R$, but not over $\Q$, that are irreducible over $\C$?''.

To construct the example, we first explain a method to build polynomials in $\Q[x] = \Q[x_1, \dots, x_n]$ that are sum of squares of polynomials with coefficients in an algebraic extension of $\Q$.
In \cite{hillar}, the author shows that if $\Q(\alpha)$ is a totally real extension of $\Q$ then every sum of squares with coefficients in $\Q(\alpha)$ is also a sum of squares of polynomials with coefficients in $\Q$. So, to construct a counterexample, we must look for extensions that are not totally real.
For simplicity, let $\alpha$ be a root of $Z^3-2$.

We start with three generic polynomials $p_i \in \Q(\alpha)[x]$, $1 \le i \le 3$, which we can write as
\[
p_i = a_i(x) + b_i(x) \alpha + c_i(x) \alpha^2,
\]
$a_i, b_i, c_i \in \Q[x]$.
From  the desired identity $f = p_1^2 + p_2^2 + p_3^2$, using the fact that $\alpha^3 = 2$, we derive an expression
\begin{equation}
\label{fABC}
f = A(a_1, a_2, \dots, c_3) + B(a_1, a_2, \dots, c_3) \alpha + C(a_1, a_2, \dots, c_3) \alpha^2.
\end{equation}
with
\begin{align*}
B &= 2 a_1 b_1+2 a_2 b_2+2 a_3 b_3+2 c_1^2+2 c_2^2+2 c_3^2, \\
C &= 2 a_1 c_1+2 a_2 c_2+2 a_3 c_3+b_1^2+b_2^2+b_3^2
\end{align*}
Since we want $f \in \Q[x]$, we must have $B \equiv 0$ and $C \equiv 0$. Note that the expressions for $B$ and $C$ are linear in $a_1, a_2$. If we solve the equations for these two coefficients, we get
\begin{align*}
a_1 &= \frac{2 a_3 b_2 c_3-2 a_3 b_3 c_2+b_1^2 b_2+b_2^3+b_2 b_3^2-2 c_1^2 c_2-2 c_2^3-2 c_2 c_3^2}{b_1 c_2-b_2 c_1} \\
a_2 &= -\frac{2 a_3 b_1 c_3-2 a_3 b_3 c_1+b_1^3+b_1 b_2^2+b_1 b_3^2-2 c_1^3-2 c_1 c_2^2-2 c_1 c_3^2}{b_1 c_2-b_2 c_1}
\end{align*}


For any number of polynomials over a number field of degree $3$ it is easy to derive similar rational expressions for $a_1$ and $a_2$. For larger degree extensions, the expression we can derive will involve roots of polynomials and hence the polynomials should be chosen carefully so that the expressions in the output are also rational expressions.

Plugging the expressions for $a_1$ and $a_2$ in (\ref{fABC}) and multiplying by the common denominator $b_1 c_2-b_2 c_1$, we obtain that for any $a_3, b_1, b_2, b_3, c_1, c_2, c_3 \in \Q[x]$, $(b_1 c_2-b_2 c_1)f \in \Q[x]$ is the sum of the squares of three polynomials with coefficients in $\Q(\alpha)[x]$. We call $f$ this polynomial from now on.

For our concrete example we pick arbitrary polynomials in $\Q(\alpha)[x,y,z,w]$ setting
\begin{equation}
\label{equation:subst}
b_1 = x, b_3=y, c_1=z, c_3=w, b_2 = 3x, c_2 = x+7z, a_3 = 21z.
\end{equation}

\begin{theorem}
Let
\begin{dmath*}
f(x,y,z,w)  =  4w^4x^2+56w^4xz+200w^4z^2-504w^3x^2z-3696w^3xz^2
 -  112w^2x^4-656w^2x^3z-12w^2x^2y^2+168w^2x^2yz+20008w^2x^2z^2
 -  88w^2xy^2z+2352w^2xyz^2+11200w^2xz^3
 +  8400w^2yz^3+20000w^2z^4+16wx^4y+7896wx^4z+128wx^3yz-10752wx^3z^2
 +  840wx^2y^2z-10328wx^2yz^2-76944wx^2z^3
 -  77616wxyz^3-184800wxz^4+932x^6-1656x^5z+188x^4y^2-2352x^4yz-10020x^4z^2
 -  256x^3y^2z- 13776x^3yz^2+3520x^3z^3+10x^2y^4-252x^2y^3z-68x^2y^2z^2
 +  49728x^2yz^3+175824x^2z^4-1848xy^3z^2+20296xy^2z^3
 +  235200xyz^4+420000xz^5+88200y^2z^4+420000yz^5+500000z^6 \in \Q[x,y,z,w].
\end{dmath*}
The polynomial $f$ can be decomposed as a sum of squares of three polynomials in $\Q(\alpha)[x,y,z,w]$, $\alpha$ a root of $Z^3-2$, but it cannot be decomposed as a sum of squares of polynomials in $\Q[x,y,z,w]$. Moreover, $f$ is irreducible over $\C$.
\end{theorem}
\begin{proof}
(See \cite[Worksheet D]{rationalSOS} for the computations in Maple.)
This polynomial was computed as explained above, using the substitutions in (\ref{equation:subst}), and it is therefore a sum of squares of three polynomials in $\Q(\alpha)[x,y,z,w]$.

To prove that it cannot be written as a sum of rational squares, we can apply the implementation in Maple of  Algorithm \ref{mainAlgorithm} (see \cite[Example 3]{rationalSOS}), which outputs that the only positive semidefinite solution to the problem has non-rational entries. For convenience, we provide next a step by step proof that the only solution for $Q$ in (\ref{SDP}) being positive semidefinite is the one corresponding to the solution in $\Q(\alpha)[x,y,z,w]$ described above.

The corresponding matrix $Q$ is a $20 \times 20$ matrix with rows and columns indexed by the monomials $v(x,y,z,w) = (x^3, x^2 y, x y^2, y^3, x^2 z, x y z, y^2 z, x z^2, y z^2, z^3, x^2 w, x y w, \allowbreak y^2 w, x z w, y z w, z^2 w, x w^2, y w^2, z w^2, w^3)$ . The space $\eLL$ has dimension 126 and rank 20.

We first find real roots of $f$. In this case, equating $f$ and its partial derivatives to $0$ we get a system that cannot be solved by Maple command \texttt{solve}. We add the equations $\{p_1 =0, p_2 = 0, p_3 = 0\}$, which must also be verified by all real solutions, and for the new system we get 4 branches of solutions:
\begin{flalign*}
\left\{x = s, y = \frac{\alpha(s,t)}{2}, z = -\frac{s}{4}, w = t\right\}, \\
\left\{x = 0, y = s, z = t, w = \beta(s,t)\right\}, \\
\left\{x = 0, y = s, z = 0, w = t\right\}, \\
\left\{x = s, y = -\frac{(100\delta^2\xi^4-441\delta^2\xi^3+28\delta\xi^4+2\xi^4+1764\delta^2-20\xi^2)s}{84\delta\xi}, \right. \\
\left. z = \delta s, w = \frac{(100\delta^2\xi^2-441\delta^2\xi+28\delta\xi^2+2\xi^2-20)s}{84\delta}\right\},
\end{flalign*}
where $\alpha(s,t)$ is a root of $m_1(Z)=4Z^2-21sZ+8t^2-168st+165s^2$, $\beta(s,t)$ is a root of $m_2(Z)=Z^2+21uv+50t^2$, $\xi$ is a root of $m_3(Z)=Z^3-2$  and $\delta$ is a root of a polynomial $m_4(Z)$ of degree $12$ with 2 real roots. Note that the first and second branches involve algebraic extensions that depend on the choices of the parameters, while the third and fourth equation do not.

We start with the third branch. By giving different values to $s$ and $t$ we obtain different vectors that must be in the kernel of $\eLL$. Note that there are 4 different monomials in $v(0,s,0,t) = (0, 0, 0, s^3, 0, 0, 0, 0, 0, 0, 0, 0, s^2t, 0, 0, 0, 0, st^2, 0, t^3)$, so by giving different values to $s$ and $t$ we can get at most 4 independent vectors. We use the following choices $\{s = 1, t = 1\}$, $\{s = 1, t = 0\}$, $\{s = 0, t = 1\}$, $\{s = -1, t = 1\}$. After solving the corresponding equations, we get $\eLL_1 \subset \eLL$ of dimension 71 and rank 16.

Next, we plugin the equations corresponding to the ghost solutions of zeros in the diagonal and $2 \times 2$ principal submatrices with null determinant. We reduce to $\eLL_2 \subset \eLL_1$ of dimension 29 variables and rank 12.

Now we use the first branch. We have to choose real values of $s$ and $t$ that give real solutions, that is, that the polynomial $4Z^2-21sZ+165s^2-168st+8t^2$ has real roots. We choose $\{s=1, t= 1\}$, $\{s=1, t = 2\}$, $\{s=1, t = 3\}$. This step reduces the space to $\eLL_3 \subset \eLL_2$ of dimension 8 and rank 9. Other choices of $s$ and $t$ tested did not reduce further the dimension.

For the second branch, we take $\{s=-3, t = 1\}$, and reduce the space to $\eLL_4 \subset \eLL_3$ of dimension 6 and rank 8. Other choices of $s$ and $t$ tested did not reduce further the dimension.

Plugging again the equations corresponding to ghost solutions from null entries in the diagonal, we obtain $\eLL_5 \subset \eLL_4$ of dimension 3 and rank 7.

Finally we use the fourth branch, which depends only on $s$. We take $s = 1$. After plugging in this solution, the system is completely solved and an exact solution is found. The solution has non-rational entries, and corresponds to the same matrix that is obtained from the original $\R$-SOS decomposition. By Lemma \ref{SOS-SDP}, we conclude that $f$ does not allow a $\Q$-SOS decomposition.

We verify that $f$ is absolutely irreducible (that is, irreducible over $\C$) using Maple procedure \texttt{evala(AIrreduc(f))}.
\end{proof}

\begin{remark}
   To prove the theorem we used only exact computations, we did not in this case need to call the numerical solver SEDUMI. The zeros of $f$ found by Maple can be easily verified to be correct by plugging them in the polynomial $f$.
\end{remark}

\appendix
\section{\texttt{rationalSOS} package usage}
\label{appendix:usage}
For the computations in this work, we developed a Maple package \texttt{rationalSOS} \cite{rationalSOS}, which is freely available for download and use. The main procedure for computing an exact SOS decomposition is \texttt{exactSOS}. We provide in this Appendix some examples of usage.

\subsection*{Example A.1}
To compute a decomposition of the polynomial given in Example \ref{exampleXYZ}, we use the following code (we assume the file \texttt{rationalSOS.mpl} is in the current working directory):

\begin{verbatim}
read("rationalSOS.mpl");
with(rationalSOS);
p1 := x^2 + 3*x*y - 5*x*z + 2z^2;
p2 := 3x^2 - 2*x*z + y*z + 5*y^2;
f := expand(p1^2 + p2^2);
out := exactSOS(f);
--> "Facial reduction results:"
--> "Original matrix"
--> " - Rank: ", 6,
--> " - Number of indeterminates: ", 6
--> "Matrix after facial reduction"
--> " - Rank: ", 2
--> " - Number of indeterminates: ", 0
--> "An exact solution was found without calling the numerical
    solver. The solution matrix is unique under the specified
    conditions."
out[1];
--> [10, 81/10]
out[2];
--> [x^2+(3/10)*x*y-(11/10)*x*z+(3/2)*y^2+(3/10)*y*z+(1/5)*z^2,
    x*y-(13/9)*x*z-(5/9)*y^2-(1/9)*y*z+(2/3)*z^2]
\end{verbatim}

The first element of the output is the list of coefficients of the polynomials in the rational SOS decomposition and the second element is the list of polynomials. In this case, the decomposition consists of two polynomials. We can verify the solution:
\begin{verbatim}
p := out[1][1]*out[2][1]^2+out[1][2]*out[2][2]^2;
expand(p);
--> 10*x^4+6*x^3*y-22*x^3*z+39*x^2*y^2-24*x^2*y*z+33*x^2*z^2-
    20*x*y^2*z+8*x*y*z^2-20*x*z^3+25*y^4+10*y^3*z+y^2*z^2+4*z^4
expand(f-p);
--> 0
\end{verbatim}

The third element in the output (not printed) is a positive semidefinite matrix satisfying (\ref{SDP}). In this case, the decomposition was obtained by purely algebraic methods, by first computing an exact real solution of $f = 0$. This implies that the matrix is the only rational matrix satisfying (\ref{SDP}).

\subsection*{Example A.2}
To compute a decomposition of the polynomial given in Example \ref{exampleB}, we use option \texttt{forceRational = "yes"} that implements the strategy explained in Section \ref{subsection:rational}. 

\begin{verbatim}
read("rationalSOS.mpl");
with(rationalSOS);
f(x,y,z) = 3618*x^8 + 468*x^7*y + 6504*x^7*z - 1104*x^6*y^2
+ 2616*x^6*y*z + 57481*x^6*z^2 - 144*x^5*y^3 - 1652*x^5*y^2*z
- 16440*x^5*y*z^2 + 23420*x^5*z^3 + 160*x^4*y^4 + 1392*x^4*y^3*z
- 2520*x^4*y^2*z^2 - 28448*x^4*y*z^3 + 91320*x^4*z^4 - 240*x^3*y^4*z
+ 1728*x^3*y^3*z^2 + 10524*x^3*y^2*z^3 - 85500*x^3*y*z^4
+ 34740*x^3*z^5 - 3696*x^2*y^3*z^3 + 28920*x^2*y^2*z^4
- 15192*x^2*y*z^5 - 57267*x^2*z^6 + 720*x*y^4*z^3 - 3312*x*y^3*z^4
- 3168*x*y^2*z^5 + 26352*x*y*z^6 - 40176*x*z^7 + 720*y^4*z^4
+ 864*y^3*z^5 - 9072*y^2*z^6 + 46656*z^8
out := exactSOS(f, forceRational = "yes");
--> "Facial reduction results:"
--> "Original matrix"
--> " - Rank: ", 15, "
--> " - Number of indeterminates: ", 75
--> "Matrix after facial reduction"
--> " - Rank: ", 6
--> " - Number of indeterminates: ", 4
--> "Calling numerical solver SEDUMI to find the values of the
    remaining indeterminates..."
--> "Problem solved. Positive definite matrix found for the reduced
    problem."
f - expand(add(i, i = out[1]*~out[2]*~out[2]));
--> 0
\end{verbatim}

The command \verb!add(i, i = out[1]*~out[2]*~out[2])! computes the sum of squares obtained by the procedure $\sum_{i=1}^6 c_i g_i^2$ and we verify that it is equal to $f$.

\subsection*{Example A.3}
We can easily verify that the Motzkin polynomial does not allow a SOS decomposition. Since we are looking for any real SOS decomposition and trace equations are only valid when looking for rational SOS, we set option \texttt{traceEquations = "no"}.

\begin{verbatim}
read("rationalSOS.mpl");
with(rationalSOS);
f := x^4*y^2 + x^2*y^4 - 3*x^2*y^2*z^2+z^6;
exactSOS(f, traceEquations = "no");
--> "Facial reduction results:"
--> "Original matrix"
--> " - Rank: ", 10
--> " - Number of indeterminates: ", 27
--> "Matrix after facial reduction"
--> " - Rank: ", 4
--> " - Number of indeterminates: ", 0
--> "An exact solution was found without calling the numerical
    solver. The solution matrix is unique under the specified
    conditions."
--> "The solution is not positive semidefinite. A SOS decomposition
    does not exist under the specified conditions."
\end{verbatim}

\subsection*{Example A.4}
We can also verify that the concrete example given in \cite[Section 2]{scheiderer} does not allow a rational SOS decomposition.

\begin{verbatim}
read("rationalSOS.mpl");
with(rationalSOS);
f:= x^4+x*y^3+y^4-3*x^2*y*z-4*x*y^2*z+2*x^2*z^2+x*z^3+y*z^3+z^4;
exactSOS(f);
--> "Option eqTrace: yes - Only valid when looking for rational
    decompositions."
--> "Facial reduction results:"
--> "Original matrix"
--> " - Rank: ", 6, "
--> " - Number of indeterminates: ", 6
--> "Matrix after facial reduction"
--> " - Rank: ", 5
--> " - Number of indeterminates: ", 0
--> "An exact solution was found without calling the numerical
--> solver. The solution matrix is unique under the specified
    conditions."
--> "The solution is not positive semidefinite. A SOS decomposition
    does not exist under the specified conditions."
\end{verbatim}

\section*{Acknowledgments}
The author would like to thank Gabriela Jeronimo and Daniel Perrucci for many fruitful discussions.

\bibliographystyle{amsplain}
\bibliography{rationalSOS}

\end{document}